\documentclass[12pt,reqno]{amsart}
\usepackage{amsmath,amssymb,amsfonts,amscd,latexsym,amsthm,mathrsfs,verbatim}
\newtheorem{lemma}{Lemma}
\newtheorem{theorem}{Theorem}

\theoremstyle{remark}

\newtheorem{example}{\bf Example}
\newtheorem*{acknowledgements}{\bf Acknowledgements}
\let\wh\widehat
\let\wt\widetilde
\def\mm{\mathrm{m}}
\renewcommand{\d}{{\mathrm d}}
\newcommand{\Cl}{\operatorname{Cl}}
\renewcommand{\Im}{\operatorname{Im}}
\renewcommand{\Re}{\operatorname{Re}}
\begin{document}

\title{Regulator of modular units and~Mahler~measures}

\author{Wadim Zudilin}
\address{School of Mathematical and Physical Sciences,
The University of Newcastle, Callaghan, NSW 2308, AUSTRALIA}
\email{wzudilin@gmail.com}

\thanks{Work is supported by the Australian Research Council.}

\date{14 April 2013. \emph{Revised}: 23 September 2013}

\begin{abstract}
We present a proof of the formula, due to Mellit and Brunault, which evaluates an integral of
the regulator of two modular units to the value of the $L$-series of a modular form of weight~2 at $s=2$.
Applications of the formula to computing Mahler measures are discussed.
\end{abstract}


\maketitle

\section{Introduction}
\label{s1}

The work of C.~Deninger \cite{De97}, D.~Boyd \cite{B98}, F.~Rodriguez-Villegas \cite{RV99} and others
provided us with a natural link between the (logarithmic) Mahler measures
\begin{equation*}
\mm\bigl(P(x_1,\dots,x_m)\bigr)
:=\frac1{(2\pi i)^m}\idotsint\limits_{|x_1|=\dots=|x_m|=1}
\log|P(x_1,\dots,x_m)|\,\frac{\d x_1}{x_1}\dotsb\frac{\d x_m}{x_m}
\end{equation*}
of certain (Laurent) polynomials $P(x_1,\dots,x_m)$, higher regulators
and Be\u\i linson's conjectures, though it took a while for those original ideas to become proofs
of some conjectural evaluations of Mahler measures. In this note we mainly discuss a recent general
formula for the regulator of two modular units due to A.~Mellit and F.~Brunault, its consequences for
2-variable Mahler measures and some related problems.

\medskip
For a smooth projective curve $C$ given as the zero locus of a polynomial $P(x,y)\in\mathbb C[x,y]$
and two rational non-constant functions $g$ and $h$ on~$C$, define the 1-form
\begin{equation}
\eta(g,h):=\log|g|\,\d\arg h-\log|h|\,\d\arg g;
\label{eta}
\end{equation}
here $\d\arg g$ is globally defined as $\Im(\d g/g)$. The form \eqref{eta} is a real 1-form defined
and infinitely many times differentiable on $C\setminus S$, where $S$ is the set of zeros and poles
of $g$ and~$h$. Furthermore, it is not hard to verify that the form~\eqref{eta} is antisymmetric,
bi-additive and closed; the latter fact follows from
$$
\d\eta(g,h)=\Im\biggl(\frac{\d g}g\wedge\frac{\d h}h\biggr)=0,
$$
as the curve $C$ has dimension~1. In turn, the closedness of~\eqref{eta} implies that,
for a closed path $\gamma$ in $C\setminus S$, the regulator map
\begin{equation}
r(\{g,h\})\colon\gamma\mapsto\int_\gamma\eta(g,h)
\label{reg}
\end{equation}
only depends on the homology class $[\gamma]$ of $\gamma$ in $H_1(C\setminus S,\mathbb Z)$.

Assuming that the polynomial $P(x,y)$ is tempered \cite{Be04,RV99},
factorising it as a polynomial in~$y$ with coefficients from $\mathbb C[x]$,
$$
P(x,y)=a_0(x)\prod_{j=1}^n(y-y_j(x)),
$$
and applying Jensen's formula, we can write \cite{Be04,Co04,LR07,RV99} the Mahler measure of~$P$ in the form
\begin{equation}
\mm\bigl(P(x,y)\bigr)
=\mm\bigl(a_0(x)\bigr)+\frac1{2\pi}\,r(\{x,y\})([\gamma]),
\label{mm-reg}
\end{equation}
where
\begin{equation}
\gamma:=\bigcup_{j=1}^n\bigl\{(x,y_j(x)):|x|=1,\;|y_j(x)|\ge1\bigr\}
=\{(x,y)\in C:|x|=1,\;|y|\ge1\}
\label{gamm}
\end{equation}
is the union of at most $n$ closed paths in $C\setminus S$.

In case the curve $C:P(x,y)=0$ admits a parameterisation by means of modular \emph{units} $x(\tau)$ and $y(\tau)$,
where the modular parameter $\tau$ belongs to the upper halfplane $\mathbb H=\{\tau\in\mathbb C:\Im\tau>0\}$,
one can change to the variable $\tau$ in the integral \eqref{reg} for $r(\{x,y\})$;
the class $[\gamma]$ in this case~\cite{Br08}
becomes a union of paths joining certain cusps of the modular functions $x(\tau)$ and $y(\tau)$.
The following general result completes the computation of the Mahler measure in the case when
$x(\tau)$ and $y(\tau)$ are given as quotients/products of modular units
\begin{gather}
g_a(\tau):=q^{NB(a/N)/2}\prod_{\substack{n\ge1\\n\equiv a\bmod N}}(1-q^n)\prod_{\substack{n\ge1\\n\equiv-a\bmod N}}(1-q^n),
\quad q=\exp(2\pi i\tau),
\label{mb01}
\\
\text{where}\quad B(x)=B_2(x):=\{x\}^2-\{x\}+\tfrac16.
\nonumber
\end{gather}

\begin{theorem}[Mellit--Brunault \cite{Me12}]
\label{MB}
For $a$, $b$ and $c$ integral, with $ac$ and $bc$ not divisible by~$N$,
\begin{equation}
\int_{c/N}^{i\infty}\eta(g_a,g_b)=\frac1{4\pi}\,L(f(\tau)-f(i\infty),2),
\label{mainmb}
\end{equation}
where the weight $2$ modular form $f(\tau)=f_{a,b;c}(\tau)$ is given by
\begin{equation*}
f_{a,b;c}:=e_{a,bc}e_{b,-ac}-e_{a,-bc}e_{b,ac}
\end{equation*}
and
\begin{equation}
e_{a,b}(\tau):=\frac12\biggl(\frac{1+\zeta_N^a}{1-\zeta_N^a}+\frac{1+\zeta_N^b}{1-\zeta_N^b}\biggr)
+\sum_{m,n\ge1}(\zeta_N^{am+bn}-\zeta_N^{-(am+bn)})q^{mn},
\quad \zeta_N:=\exp(2\pi i/N),
\label{eab}
\end{equation}
are weight $1$ level $N^2$ Eisenstein series.
\end{theorem}

The $L$-value on the right-hand side of \eqref{mainmb} is well defined because
of subtracting the constant term
\begin{align*}
f(i\infty)
&=\frac12\biggl(\frac{1+\zeta_N^{b}}{1-\zeta_N^{b}}\,\frac{1+\zeta_N^{bc}}{1-\zeta_N^{bc}}
-\frac{1+\zeta_N^{a}}{1-\zeta_N^{a}}\,\frac{1+\zeta_N^{ac}}{1-\zeta_N^{ac}}\biggr)
\\
&=-\frac12\biggl(\cot\frac{\pi b}N\,\cot\frac{\pi bc}N-\cot\frac{\pi a}N\,\cot\frac{\pi ac}N\biggr)
\end{align*}
in the $q$-expansion $f(\tau)=f(i\infty)+\sum_{n\ge1}c_nq^n$.
Furthermore, if a linear combination
$$
f(\tau)=\sum_{(a,b,c)\in\mathcal M}\lambda_{a,b,c}f_{a,b;c}(\tau),
\qquad \lambda_{a,b,c}\in\mathbb C,
$$
happens to be a \emph{cusp} form (and this corresponds to application of Theorem~\ref{MB} to Mahler measures),
then formula \eqref{mainmb} produces the evaluation
\begin{equation*}
\sum_{(a,b,c)\in\mathcal M}\lambda_{a,b,c}\int_{c/N}^{i\infty}\eta(g_a,g_b)=\frac1{4\pi}\,L(f(\tau),2).
\end{equation*}
Note as well that the theorem allows one to integrate between any cusps $c/N$ and $d/N$ with the help of
$\int_{c/N}^{d/N}=\int_{c/N}^{i\infty}-\int_{d/N}^{i\infty}$.

Here is a sketch of the proof of Theorem~\ref{MB}; details are given in Section~\ref{s2}.
We parameterise the contour of integration by
$\tau=c/N+it$, $0<t<\infty$, and note that the M\"obius transformation
$\tau':=(c\tau-(c^2+1)/N)/(N\tau-c)$ preserves the contour: $\tau'=c/N+i/(N^2t)$.
Then the logarithms of $g_a(\tau)$ and $g_b(\tau)$, hence their real and imaginary parts\,---\,everything
we need for computing the form \eqref{eta}, can be written as explicit Eisenstein series of weight~0
in powers of $\exp(-2\pi t)$ and $\exp(-2\pi/(N^2t))$.
Finally, executing an analytical change of variable from~\cite{RZ12} (as detailed
in \cite[Section~3]{Zu13}) the integrand becomes a linear
combination of pairwise products of weight~1 Eisenstein series in powers of $\exp(-2\pi t)$
integrated against the form $t\,\d t$ along the line $0<t<\infty$.

\medskip
Applications of Theorem~\ref{MB} to Boyd's and Rodriguez-Villegas' conjectural evaluations of 2-variable Mahler measures
are discussed in Section~\ref{s3}, while Section~\ref{s4} highlights some open problems related
to 3-variable Mahler measures.

\section{Proof of the Mellit--Brunault formula}
\label{s2}

The two auxiliary lemmas indicate particular modular transformations of the modular functions \eqref{mb01}
and the Eisenstein series~\eqref{eab}. Lemma~\ref{lem1} also describes the asymptotic behaviour
of the modular functions \eqref{mb01} in a neighbourhood of a cusp with $\Re\tau=0$; it is used
in the form~\eqref{cuspass} to determine the integration contours \eqref{gamm} for our applications
in Section~\ref{s3}.

\begin{lemma}
\label{lem1}
For $a,c$ integers,
\begin{align*}
\log g_a(c/N+it)
&=\pi icB(a/N)-\pi t\,NB(a/N)
\\ &\qquad
-\sum_{\substack{m,n\ge1\\n\equiv a}}\frac{\zeta_N^{acm}}m\,\exp(-2\pi mnt)
-\sum_{\substack{m,n\ge1\\n\equiv-a}}\frac{\zeta_N^{-acm}}m\,\exp(-2\pi mnt)
\\
&=-\frac{\pi i}2+\pi ia(c^2+1)(N-ac)+\pi icB(ac/N)-\frac{\pi B(ac/N)}{Nt}
\\ &\qquad
-\sum_{\substack{m,n\ge1\\n\equiv ac}}\frac{\zeta_N^{-am}}m\,\exp\biggl(-\frac{2\pi mn}{N^2t}\biggr)
-\sum_{\substack{m,n\ge1\\n\equiv-ac}}\frac{\zeta_N^{am}}m\,\exp\biggl(-\frac{2\pi mn}{N^2t}\biggr),
\end{align*}
where $t>0$.
\end{lemma}

\begin{proof}
First note that definition~\eqref{mb01} implies
\begin{align*}
\log g_a(\tau)
&=\pi i\tau\,NB(a/N)
+\sum_{\substack{n\ge1\\n\equiv a}}\log(1-q^n)
+\sum_{\substack{n\ge1\\n\equiv-a}}\log(1-q^n)
\\
&=\pi i\tau\,NB(a/N)
-\sum_{\substack{m,n\ge1\\n\equiv a}}\frac{q^{mn}}m
-\sum_{\substack{m,n\ge1\\n\equiv-a}}\frac{q^{mn}}m.
\end{align*}
Therefore, the substitution $\tau=c/N+it$, equivalently $q=\zeta_N^c\exp(-2\pi t)$, results in the first expansion of the lemma.

Secondly, the modular units \eqref{mb01} are particular cases of the `generalized Dedekind eta functions' \cite[eq.~(3)]{Ya04}.
Applying \cite[Theorem~1]{Ya04} with the choice $h=0$ and
$\gamma=\bigl(\begin{smallmatrix} c & -c^2-1 \\ 1 & -c \end{smallmatrix}\bigr)$ we deduce that
$$
g_a(\tau)=\wt g_{a,c}\biggl(\frac{c\tau-(c^2+1)/N}{N\tau-c}\biggr),
$$
where
\begin{multline*}
\wt g_{a,c}(\tau):=\exp(-\pi i/2+\pi ia(c^2+1)(N-ac))\,q^{NB(ac/N)/2}
\\ \times
\prod_{\substack{n\ge1\\n\equiv ac\bmod N}}(1-\zeta_N^{-a(c^2+1)}q^n)\prod_{\substack{n\ge1\\n\equiv-ac\bmod N}}(1-\zeta_N^{a(c^2+1)}q^n).
\end{multline*}
On the other hand,
$$
\tau':=\frac{c\tau-(c^2+1)/N}{N\tau-c}\bigg|_{\tau=c/N+it}
=\frac cN+\frac i{N^2t},
$$
so that
\begin{multline*}
\log\wt g_{a,c}(\tau')
=-\frac{\pi i}2+\pi ia(c^2+1)(N-ac)+\pi icB(ac/N)-\frac{\pi B(ac/N)}{Nt}
\\
-\sum_{\substack{m,n\ge1\\n\equiv ac}}\frac{\zeta_N^{-a(c^2+1)m+cmn}}m\,\exp\biggl(-\frac{2\pi mn}{N^2t}\biggr)
-\sum_{\substack{m,n\ge1\\n\equiv-ac}}\frac{\zeta_N^{a(c^2+1)m+cmn}}m\,\exp\biggl(-\frac{2\pi mn}{N^2t}\biggr),
\end{multline*}
and it remains to use the congruences $n\equiv ac$ and $n\equiv-ac$ to simplify the exponents of the roots of unity.
\end{proof}

\begin{lemma}
\label{lem2}
For $a,b$ integers not divisible by~$N$,
$$
\frac1{N^2\tau}\,e_{a,b}\biggl(-\frac1{N^2\tau}\biggr)
=\wt e_{a,b}(\tau):=\sum_{\substack{m,n\ge1\\m\equiv a,\ n\equiv b}}q^{mn}
-\sum_{\substack{m,n\ge1\\m\equiv-a,\ n\equiv-b}}q^{mn}.
$$
\end{lemma}

\begin{proof}
In \cite[Section 7]{Sch74} the following general Eisenstein series of weight 1 and level $N$ are introduced:
$$
G_{a,c}(\tau)
=G_{N,1;(c,a)}(\tau)
:=-\frac{2\pi i}N\Biggl(\kappa_{a,c}+\sum_{\substack{m,n\ge1\\n\equiv c\bmod N}}\zeta_N^{am}q^{mn/N}
-\sum_{\substack{m,n\ge1\\n\equiv-c\bmod N}}\zeta_N^{-am}q^{mn/N}\Biggr),
$$
where
$$
\kappa_{a,c}:=\begin{cases}
\dfrac12\,\dfrac{1+\zeta_N^a}{1-\zeta_N^a} &\text{if $c\equiv0\bmod N$}, \\[1.5mm]
\dfrac12-\biggl\{\dfrac cN\biggr\} &\text{if $c\not\equiv0\bmod N$}.
\end{cases}
$$
Then for $\gamma=\bigl(\begin{smallmatrix} A&B\\C&D\end{smallmatrix}\bigr)\in SL_2(\mathbb Z)$ we have
\begin{equation}
G_{a,c}(\gamma\tau)
=(C\tau+D)G_{aD+cB,aC+cA}(\tau).
\label{tr1a}
\end{equation}

The partial Fourier transform from \cite[Chapter~III]{Ka76} applied to $G_{a,c}$ results in
\begin{align*}
\wh G_{a,b}(\tau)
:=\sum_{c=0}^{N-1}\zeta_N^{bc}\,G_{a,c}(\tau)
&=-\frac{\pi i}N\biggl(\frac{1+\zeta_N^a}{1-\zeta_N^a}+\frac{1+\zeta_N^b}{1-\zeta_N^b}\biggr)
\\ &\quad
-\frac{2\pi i}N\sum_{m,n\ge1}(\zeta_N^{am+bn}-\zeta_N^{-(am+bn)})q^{mn/N}.
\end{align*}
On the other hand, taking $\gamma=\bigl(\begin{smallmatrix} 0&-1\\1&0\end{smallmatrix}\bigr)$ in~\eqref{tr1a} we find that
\begin{align*}
\tau^{-1}\wh G_{a,b}(-1/\tau)
&=\sum_{c=0}^{N-1}\zeta_N^{bc}\,G_{-c,a}(\tau)
\\
&=-\frac{2\pi i}N\sum_{c=0}^{N-1}\zeta_N^{bc}\Biggl(\frac12-\biggl\{\frac aN\biggr\}+\sum_{\substack{m,n\ge1\\n\equiv a}}\zeta_N^{-cm}q^{mn/N}
-\sum_{\substack{m,n\ge1\\n\equiv-a}}\zeta_N^{cm}q^{mn/N}\Biggr)
\\
&=-2\pi i\Biggl(\sum_{\substack{m,n\ge1\\n\equiv a,\ m\equiv b}}q^{mn/N}
-\sum_{\substack{m,n\ge1\\n\equiv-a,\ m\equiv-b}}q^{mn/N}\Biggr).
\end{align*}
Using now $\wh G_{a,b}(N\tau)=-2\pi i\,e_{a,b}(\tau)/N$ we obtain the desired transformation.
\end{proof}

The next two statements are to take care of integrating the constant terms of auxiliary Eisenstein series.

\begin{lemma}
\label{lem3}
For $a,b$ integers not divisible by $N$,
$$
\int_0^\infty\biggl(e_{a,b}(it)+e_{a,-b}(it)-\frac{1+\zeta_N^a}{1-\zeta_N^a}\biggr)\,t\,\d t
=i\Cl_2\biggl(\frac{2\pi a}N\biggr)B\biggl(\frac bN\biggr),
$$
where
$$
\Cl_2(x):=\sum_{m\ge1}\frac{\sin mx}{m^2}
$$
denotes Clausen's \textup(dilogarithmic\textup) function.
\end{lemma}

\begin{proof}
The integral under consideration is equal to
\begin{align*}
&
\int_0^\infty\sum_{m,n\ge1}(\zeta_N^{am+bn}-\zeta_N^{-(am+bn)}+\zeta_N^{am-bn}-\zeta_N^{-(am-bn)})\exp(-2\pi mnt)\,t\,\d t
\\ &\quad
=\int_0^\infty\sum_{m,n\ge1}(\zeta_N^{am}-\zeta_N^{-am})(\zeta_N^{bn}+\zeta_N^{-bn})\exp(-2\pi mnt)\,t\,\d t.
\end{align*}
On using the Mellin transform
\begin{equation}
\int_0^\infty\exp(-2\pi kt)t^{s-1}\,\d t
=\frac{\Gamma(s)}{(2\pi)^sk^s}
\qquad\text{for}\quad \Re s>0,
\label{mellin}
\end{equation}
the integral of the double sum evaluates to
\begin{align*}
\frac1{4\pi^2}\sum_{m\ge1}\frac{\zeta_N^{am}-\zeta_N^{-am}}{m^2}\sum_{n\ge1}\frac{\zeta_N^{bn}+\zeta_N^{-bn}}{n^2}
=\frac i{\pi^2}\,\Cl_2\biggl(\frac{2\pi a}N\biggr)\sum_{n\ge1}\frac{\cos(2\pi nb/N)}{n^2}.
\end{align*}
It remains to use
$$
\sum_{n\ge1}\frac{\cos nx}{n^2}=\pi^2B\biggl(\frac x{2\pi}\biggr),
$$
and the required evaluation follows.
\end{proof}

\begin{lemma}
\label{lem4}
For $a,b$ integers not divisible by $N$,
\begin{align*}
&
\int_0^\infty\frac1{iNt}\,\d\sum_{m\ge1}\frac{\zeta_N^{am}-\zeta_N^{-am}}m
\Biggl(\sum_{\substack{n\ge1\\n\equiv b}}-\sum_{\substack{n\ge1\\n\equiv-b}}\Biggr)\exp\biggl(-\frac{2\pi mn}{N^2t}\biggr)
\\ &\quad
=-i\,\Cl_2\biggl(\frac{2\pi a}N\biggr)\,\frac{1+\zeta_N^b}{1-\zeta_N^b}.
\end{align*}
\end{lemma}

\begin{proof}
Performing the change of variable $u=1/(N^2t)$ in the integral, it becomes equal to
$$
\frac{2\pi N}i\int_0^\infty\sum_{m\ge1}(\zeta_N^{am}-\zeta_N^{-am})
\Biggl(\sum_{\substack{n\ge1\\n\equiv b}}-\sum_{\substack{n\ge1\\n\equiv-b}}\Biggr)n\exp(-2\pi mnu)\,u\,\d u,
$$
and applying \eqref{mellin} with $s\to2^+$ it evaluates to
\begin{align*}
&
\frac N{\pi}\sum_{m\ge1}\frac{\sin(2\pi am/N)}{m^2}
\lim_{s\to1^+}\Biggl(\sum_{\substack{n\ge1\\n\equiv b}}-\sum_{\substack{n\ge1\\n\equiv-b}}\Biggr)\frac1{n^s}
\\ &\quad
=\frac1{\pi}\,\Cl_2\biggl(\frac{2\pi a}N\biggr)\cdot\bigl(\psi(1-\{b/N\})-\psi(\{b/N\})\bigr)
=\frac1{\pi}\,\Cl_2\biggl(\frac{2\pi a}N\biggr)\,\pi\cot\frac{\pi b}N,
\end{align*}
where $\psi(x)$ is the logarithmic derivative of the gamma function. It remains to use
$\cot(\pi b/N)=-i(1+\zeta_N^b)/(1-\zeta_N^b)$.
\end{proof}

\begin{proof}[Proof of Theorem~\textup{\ref{MB}}]
To integrate the 1-form $\eta(g_a,g_b)$ along the interval $\tau\in(c/N,i\infty)$ we make the substitution $\tau=c/N+it$, $0<t<\infty$.
It follows from Lemma~\ref{lem1} that
\begin{align}
\log|g_a(\tau)|
&=-\frac{\pi B(ac/N)}{Nt}
-\frac12\sum_{m\ge1}\frac{\zeta_N^{am}+\zeta_N^{-am}}m
\Biggl(\sum_{\substack{n\ge1\\n\equiv ac}}+\sum_{\substack{n\ge1\\n\equiv-ac}}\Biggr)\exp\biggl(-\frac{2\pi mn}{N^2t}\biggr)
\label{cuspass}
\\ \intertext{and}
\d\arg g_a(\tau)
&=-\frac1{2i}\,\d\sum_{m\ge1}\frac{\zeta_N^{acm}-\zeta_N^{-acm}}m
\Biggl(\sum_{\substack{n\ge1\\n\equiv a}}-\sum_{\substack{n\ge1\\n\equiv-a}}\Biggr)\exp(-2\pi mnt)
\nonumber\\
&=\frac1{2i}\,\d\sum_{m\ge1}\frac{\zeta_N^{am}-\zeta_N^{-am}}m
\Biggl(\sum_{\substack{n\ge1\\n\equiv ac}}-\sum_{\substack{n\ge1\\n\equiv-ac}}\Biggr)\exp\biggl(-\frac{2\pi mn}{N^2t}\biggr).
\nonumber
\end{align}
This computation implies
\begin{align*}
\eta(g_a,g_b)
&=-\frac{\pi B(ac/N)}{2iNt}\,\d\sum_{m\ge1}\frac{\zeta_N^{bm}-\zeta_N^{-bm}}m
\Biggl(\sum_{\substack{n\ge1\\n\equiv bc}}-\sum_{\substack{n\ge1\\n\equiv-bc}}\Biggr)\exp\biggl(-\frac{2\pi mn}{N^2t}\biggr)
\\ &\qquad
+\frac1{4i}\sum_{m_1\ge1}\frac{\zeta_N^{am_1}+\zeta_N^{-am_1}}{m_1}
\Biggl(\sum_{\substack{n_1\ge1\\n_1\equiv ac}}+\sum_{\substack{n_1\ge1\\n_1\equiv-ac}}\Biggr)\exp\biggl(-\frac{2\pi m_1n_1}{N^2t}\biggr)
\\ &\qquad\quad\times
\d\sum_{m_2\ge1}\frac{\zeta_N^{bcm_2}-\zeta_N^{-bcm_2}}{m_2}
\Biggl(\sum_{\substack{n_2\ge1\\n_2\equiv b}}-\sum_{\substack{n_2\ge1\\n_2\equiv-b}}\Biggr)\exp(-2\pi m_2n_2t)
\displaybreak[2]\\ &\;
+\frac{\pi B(bc/N)}{2iNt}\,\d\sum_{m\ge1}\frac{\zeta_N^{am}-\zeta_N^{-am}}m
\Biggl(\sum_{\substack{n\ge1\\n\equiv ac}}-\sum_{\substack{n\ge1\\n\equiv-ac}}\Biggr)\exp\biggl(-\frac{2\pi mn}{N^2t}\biggr)
\\ &\;\qquad
-\frac1{4i}\sum_{m_1\ge1}\frac{\zeta_N^{bm_1}+\zeta_N^{-bm_1}}{m_1}
\Biggl(\sum_{\substack{n_1\ge1\\n_1\equiv bc}}+\sum_{\substack{n_1\ge1\\n_1\equiv-bc}}\Biggr)\exp\biggl(-\frac{2\pi m_1n_1}{N^2t}\biggr)
\\ &\;\qquad\quad\times
\d\sum_{m_2\ge1}\frac{\zeta_N^{acm_2}-\zeta_N^{-acm_2}}{m_2}
\Biggl(\sum_{\substack{n_2\ge1\\n_2\equiv a}}-\sum_{\substack{n_2\ge1\\n_2\equiv-a}}\Biggr)\exp(-2\pi m_2n_2t).
\end{align*}
The terms involving double sums only can be integrated with the help of Lemma~\ref{lem4}, and we obtain
\begin{align*}
&
\int_{c/N}^{i\infty}\eta(g_a,g_b)
=\frac{\pi i}2\,\frac{1+\zeta_N^{bc}}{1-\zeta_N^{bc}}\,\Cl_2\biggl(\frac{2\pi b}N\biggr)B\biggl(\frac{ac}N\biggr)
-\frac{\pi i}2\,\frac{1+\zeta_N^{ac}}{1-\zeta_N^{ac}}\,\Cl_2\biggl(\frac{2\pi a}N\biggr)B\biggl(\frac{bc}N\biggr)
\\ &\quad
-\frac{\pi}{2i}\Biggl(\sum_{m_1,m_2\ge1}(\zeta_N^{am_1}+\zeta_N^{-am_1})(\zeta_N^{bcm_2}-\zeta_N^{-bcm_2})
\Biggl(\sum_{\substack{n_1\ge1\\n_1\equiv ac}}+\sum_{\substack{n_1\ge1\\n_1\equiv-ac}}\Biggr)
\Biggl(\sum_{\substack{n_2\ge1\\n_2\equiv b}}-\sum_{\substack{n_2\ge1\\n_2\equiv-b}}\Biggr)
\\ &\quad\quad
-\sum_{m_1,m_2\ge1}(\zeta_N^{bm_1}+\zeta_N^{-bm_1})(\zeta_N^{acm_2}-\zeta_N^{-acm_2})
\Biggl(\sum_{\substack{n_1\ge1\\n_1\equiv bc}}+\sum_{\substack{n_1\ge1\\n_1\equiv-bc}}\Biggr)
\Biggl(\sum_{\substack{n_2\ge1\\n_2\equiv a}}-\sum_{\substack{n_2\ge1\\n_2\equiv-a}}\Biggr)\Biggr)
\\ &\quad\qquad\times
\frac{n_2}{m_1}\int_0^\infty\exp\biggl(-2\pi\biggl(\frac{m_1n_1}{N^2t}+m_2n_2t\biggr)\biggr)\,\d t.
\end{align*}
Now we execute the change of variable $u=n_2t/m_1$, interchange integration and quadruple summation
and use Lemma~\ref{lem2}:
\begin{align*}
&
\int_{c/N}^{i\infty}\eta(g_a,g_b)
=\frac{\pi i}2\,\frac{1+\zeta_N^{bc}}{1-\zeta_N^{bc}}\,\Cl_2\biggl(\frac{2\pi b}N\biggr)B\biggl(\frac{ac}N\biggr)
-\frac{\pi i}2\,\frac{1+\zeta_N^{ac}}{1-\zeta_N^{ac}}\,\Cl_2\biggl(\frac{2\pi a}N\biggr)B\biggl(\frac{bc}N\biggr)
\\ &\quad
-\frac{\pi}{2i}\int_0^\infty\sum_{m_1,m_2\ge1}(\zeta_N^{am_1}+\zeta_N^{-am_1})(\zeta_N^{bcm_2}-\zeta_N^{-bcm_2})\exp(-2\pi m_1m_2u)
\\ &\quad\qquad\times
\Biggl(\sum_{\substack{n_1\ge1\\n_1\equiv ac}}+\sum_{\substack{n_1\ge1\\n_1\equiv-ac}}\Biggr)
\Biggl(\sum_{\substack{n_2\ge1\\n_2\equiv b}}-\sum_{\substack{n_2\ge1\\n_2\equiv-b}}\Biggr)
\exp\biggl(-\frac{2\pi n_1n_2}{N^2u}\biggr)
\\ &\quad\quad
-\sum_{m_1,m_2\ge1}(\zeta_N^{bm_1}+\zeta_N^{-bm_1})(\zeta_N^{acm_2}-\zeta_N^{-acm_2})\exp(-2\pi m_1m_2u)
\\ &\quad\qquad\times
\Biggl(\sum_{\substack{n_1\ge1\\n_1\equiv bc}}+\sum_{\substack{n_1\ge1\\n_1\equiv-bc}}\Biggr)
\Biggl(\sum_{\substack{n_2\ge1\\n_2\equiv a}}-\sum_{\substack{n_2\ge1\\n_2\equiv-a}}\Biggr)
\exp\biggl(-\frac{2\pi n_1n_2}{N^2u}\biggr)\,\d u
\displaybreak[2]\\ &\;
=\frac{\pi i}2\,\frac{1+\zeta_N^{bc}}{1-\zeta_N^{bc}}\,\Cl_2\biggl(\frac{2\pi b}N\biggr)B\biggl(\frac{ac}N\biggr)
-\frac{\pi i}2\,\frac{1+\zeta_N^{ac}}{1-\zeta_N^{ac}}\,\Cl_2\biggl(\frac{2\pi a}N\biggr)B\biggl(\frac{bc}N\biggr)
\\ &\quad
-\frac{\pi}{2i}\int_0^\infty\biggl(e_{a,bc}(iu)-e_{a,-bc}(iu)-\frac{1+\zeta_N^{bc}}{1-\zeta_N^{bc}}\biggr)
\bigl(\wt e_{b,ac}(i/(N^2u))+\wt e_{b,-ac}(i/(N^2u))\bigr)
\\ &\quad\quad
-\biggl(e_{b,ac}(iu)-e_{b,-ac}(iu)-\frac{1+\zeta_N^{ac}}{1-\zeta_N^{ac}}\biggr)
\bigl(\wt e_{a,bc}(i/(N^2u))+\wt e_{a,-bc}(i/(N^2u))\bigr)\,\d u
\displaybreak[2]\\ &\;
=\frac{\pi i}2\,\frac{1+\zeta_N^{bc}}{1-\zeta_N^{bc}}\,\Cl_2\biggl(\frac{2\pi b}N\biggr)B\biggl(\frac{ac}N\biggr)
-\frac{\pi i}2\,\frac{1+\zeta_N^{ac}}{1-\zeta_N^{ac}}\,\Cl_2\biggl(\frac{2\pi a}N\biggr)B\biggl(\frac{bc}N\biggr)
\\ &\quad
+\frac{\pi}{2}\int_0^\infty\biggl(e_{a,bc}(iu)-e_{a,-bc}(iu)-\frac{1+\zeta_N^{bc}}{1-\zeta_N^{bc}}\biggr)
\bigl(e_{b,ac}(iu)+e_{b,-ac}(iu)\bigr)\,u
\\ &\quad\quad
-\biggl(e_{b,ac}(iu)-e_{b,-ac}(iu)-\frac{1+\zeta_N^{ac}}{1-\zeta_N^{ac}}\biggr)
\bigl(e_{a,bc}(iu)+e_{a,-bc}(iu)\bigr)\,u\,\d u
\displaybreak[2]\\ &\;
=\frac{\pi i}2\,\frac{1+\zeta_N^{bc}}{1-\zeta_N^{bc}}\,\Cl_2\biggl(\frac{2\pi b}N\biggr)B\biggl(\frac{ac}N\biggr)
-\frac{\pi i}2\,\frac{1+\zeta_N^{ac}}{1-\zeta_N^{ac}}\,\Cl_2\biggl(\frac{2\pi a}N\biggr)B\biggl(\frac{bc}N\biggr)
\\ &\quad
+\pi\int_0^\infty\bigl(e_{a,bc}(iu)e_{b,-ac}(iu)-e_{a,-bc}(iu)e_{b,ac}(iu)\bigr)\,u
\\ &\quad\quad
-\frac12\biggl(\frac{1+\zeta_N^{bc}}{1-\zeta_N^{bc}}\bigl(e_{b,ac}(iu)+e_{b,-ac}(iu)\bigr)
-\frac{1+\zeta_N^{ac}}{1-\zeta_N^{ac}}\bigl(e_{a,bc}(iu)+e_{a,-bc}(iu)\bigr)\biggr)u\,\d u
\displaybreak[2]\\ \intertext{(we apply Lemma~\ref{lem3})}
&\;
=\pi\int_0^\infty\biggl(f_{a,b;c}(iu)
+\frac12\,\frac{1+\zeta_N^{a}}{1-\zeta_N^{a}}\,\frac{1+\zeta_N^{ac}}{1-\zeta_N^{ac}}
-\frac12\,\frac{1+\zeta_N^{b}}{1-\zeta_N^{b}}\,\frac{1+\zeta_N^{bc}}{1-\zeta_N^{bc}}\biggr)\,u\,\d u,
\end{align*}
and the result follows by appealing to~\eqref{mellin}.
\end{proof}

\section{Applications}
\label{s3}

The modularity theorem guarantees that an \emph{elliptic} curve $C:P(x,y)=0$ can be parameterised
by modular functions $x(\tau)$ and $y(\tau)$, whose level~$N$ is necessarily the conductor of~$C$,
such that the pull-back of the canonical differential on $C$ is proportional to $2\pi if(\tau)\,\d\tau=f(\tau)\,\d q/q$,
where $f$ is (up to an isogeny) a normalised newform of weight 2 and level~$N$, which automatically
happens to be a cusp form and a Hecke eigenform. Computing the conductor of~$C$ and producing the
cusp form $f$ of this level give an efficient strategy to determine successively the coefficients
in the $q$-expansions of $x(\tau)=\varepsilon_1q^{-M_1}+\dotsb$ and $y(\tau)=\varepsilon_2q^{-M_2}+\dotsb$
subject to $P(x(\tau),y(\tau))=0$, where $\varepsilon_1$ and $\varepsilon_2$ are suitable nonzero constants.
The particular form of $q$-expansions only fixes a normalisation of $x(\tau)$ and $y(\tau)$
up to the action of the corresponding congruence subgroup $\Gamma_0(N)$.
Finally, it remains to verify whether $x(\tau)$ and $y(\tau)$ just found are modular units\,---\,modular functions
whose all zeroes and poles are at cusps (so that they admit eta-like product expansions);
if this is the case, we can use Theorem~\ref{MB} to compute the Mahler measure $\mm(P(x,y))$.
Note that the property of being a modular unit imposes a strong condition on the $q$-expansion
of the logarithmic derivative\,---\,it can be easily detected in practice by examining a couple of (hundred) terms
in the $q$-expansion of the latter.

In this section we touch the  `classical' family of Mahler measures
$$
\mm(xy^2+(x^2+kx+1)y+x)=\mm\Bigl(k+x+\frac1x+y+\frac1y\Bigr), \qquad k^2\in\mathbb Z\setminus\{0,16\},
$$
which goes back to the works \cite{B98,De97,RV99}. Namely, we will see that Theorem~\ref{MB} applies
in the cases when the corresponding zero locus
\begin{equation}
E:k+x+\frac1x+y+\frac1y=0
\label{mb02}
\end{equation}
can be parameterised by modular units.
For this family of tempered Laurent polynomials, equation \eqref{mm-reg} assumes the form
\begin{equation}
\mm\Bigl(k+x+\frac1x+y+\frac1y\Bigr)
=\mm(y^2+(k+x+x^{-1})y+1)
=\frac1{2\pi}\,r(\{x,y\})([\gamma]),
\label{mm-reg1}
\end{equation}
where $\gamma$ is a single closed path on $E\setminus\{(0,0)\}$ corresponding to the zero
$y_1(x)$ of $y^2+(k+x+x^{-1})y+1$ which satisfies $|y_1(x)|\ge1$.

The above general strategy restricted to the family~\eqref{mb02} was identified by Mellit in~\cite{Me11}
and illustrated by him on the example of $k=2i$; this is Example~\ref{ex2} below. The modular functions
$x$ and $y$ satisfying \eqref{mb02} are searched in the form $x(\tau)=(\varepsilon q)^{-1}+\dotsb$
and $y(\tau)=-(\varepsilon q)^{-1}+\dotsb$, where $\varepsilon\in\mathbb Z[k]$ is chosen so that
$k/\varepsilon$~is a positive integer. The condition on the pull-back of the canonical differential on~$E$ takes the form
\begin{equation*}
\frac{q\,(\d x/\d q)}{\varepsilon x(y-1/y)}=f,
\end{equation*}
where $f(\tau)$ is the corresponding Hecke eigenform of weight~2.

The computational part of the examples below was accomplished in \texttt{sage} and \texttt{gp-pari},
which allowed us to compute as many terms in the $q$-expansions of a modular parameterisation of a given elliptic curve as requested.
Assisted with this software, we were normally able to relate occurring \emph{modular} forms and functions
(for example, their product expansions) by computing and examining sufficiently many terms in their $q$-expansions.

Below we will have
occasional appearance of Dedekind's eta-function $\eta(\tau):=q^{1/24}\prod_{n=1}^\infty(1-q^n)$. We hope that this
extra eta notation does not cause any confusion with~\eqref{eta}, as it depends here on a \emph{single} variable,
which is always a rational multiple of $\tau$ from the upper halfplane.

\begin{example}
\label{ex1}
The most classical example corresponds to the choice $k=1$,
when the elliptic curve in~\eqref{mb02} has conductor $N=15$ and
can be parameterised by modular units
\begin{align*}
x(\tau)
&=\phantom-\frac1q\prod_{n=0}^\infty\frac{(1-q^{15n+7})(1-q^{15n+8})}{(1-q^{15n+2})(1-q^{15n+13})}
=\phantom-\frac{g_7(\tau)}{g_2(\tau)},
\\
y(\tau)
&=-\frac1q\prod_{n=0}^\infty\frac{(1-q^{15n+4})(1-q^{15n+11})}{(1-q^{15n+1})(1-q^{15n+14})}
=-\frac{g_4(\tau)}{g_1(\tau)},
\end{align*}
so that
$$
\frac{q\,(\d x/\d q)}{x(y-1/y)}=f_{15}(\tau):=\eta(\tau)\eta(3\tau)\eta(5\tau)\eta(15\tau)
$$
and the path of integration $\gamma$ in \eqref{mm-reg1} corresponds to the range of $\tau$ between
the two cusps $-1/5$ and $1/5$ of $\Gamma_0(15)$. Therefore, Theorem~\ref{MB} results in
\begin{align*}
\mm\Bigl(1+x+\frac1x+y+\frac1y\Bigr)
&=\frac1{2\pi}\Bigl(\int_{-1/5}^{i\infty}-\int_{1/5}^{i\infty}\Bigr)\eta(g_7/g_2,g_4/g_1)
\\
&=\frac1{8\pi^2}L(2f_{7,4;-3}-2f_{7,1;-3}-2f_{2,4;-3}+2f_{2,1;-3},2)
\\
&=\frac{15}{4\pi^2}L(f_{15},2),
\end{align*}
which is precisely Boyd's conjecture from~\cite{B98} first proven in~\cite{RZ13}.

Note that this evaluation implies some other Mahler measures, namely \cite{La10,LR07}
\begin{align*}
\mm\Bigl(5+x+\frac1x+y+\frac1y\Bigr)
&=\phantom06\mm\Bigl(1+x+\frac1x+y+\frac1y\Bigr)
\\
\mm\Bigl(16+x+\frac1x+y+\frac1y\Bigr)
&=11\mm\Bigl(1+x+\frac1x+y+\frac1y\Bigr),
\\
\mm\Bigl(3i+x+\frac1x+y+\frac1y\Bigr)
&=\phantom05\mm\Bigl(1+x+\frac1x+y+\frac1y\Bigr),
\end{align*}
though the corresponding elliptic curves
$k+x+1/x+y+1/y=0$ for $k=5$, $16$ and $3i$
are not parameterised by modular units.
\end{example}

\begin{example}[{\cite{Me11}}]
\label{ex2}
The modular parameterisation of \eqref{mb02} for $k=2i$ (the conductor of
elliptic curve is then $N=40$) and the corresponding Mahler measure evaluation
$$
\mm\Bigl(2i+x+\frac1x+y+\frac1y\Bigr)=\frac{10}{\pi^2}L(f_{40},2),
$$
where
$$
f_{40}(\tau):=\frac{\eta(\tau)\eta(8\tau)\eta(10\tau)^2\eta(20\tau)^2}{\eta(5\tau)\eta(40\tau)}
+\frac{\eta(2\tau)^2\eta(4\tau)^2\eta(5\tau)\eta(40\tau)}{\eta(\tau)\eta(8\tau)},
$$
were given in Mellit's talk~\cite{Me11}. He identifies $x(\tau)$ and $y(\tau)$ with
infinite products which are fully expressible by means of Ramanujan's lambda function
$$
\lambda(\tau)=q^{1/5}\prod_{n=1}^\infty(1-q^n)^{\left(\frac n5\right)}
=q^{1/5}\prod_{n=1}^\infty\frac{(1-q^{5n-1})(1-q^{5n-4})}{(1-q^{5n-2})(1-q^{5n-3})};
$$
namely,
\begin{align*}
x(\tau)&=-i\,\frac{\lambda(4\tau)}{\lambda(\tau)\lambda(8\tau)}
=-i\,\frac{g_2g_3g_7g_{13}g_{16}g_{17}g_{18}}{g_1g_6g_8g_9g_{11}g_{14}g_{19}},
\\
y(\tau)
&=\phantom-i\,\frac{\lambda(\tau)\lambda(2\tau)}{\lambda(8\tau)}
=\phantom-i\,\frac{g_1g_9g_{11}g_{16}g_{19}}{g_3g_7g_8g_{13}g_{17}}
\end{align*}
in the notation~\eqref{mb01} with $N=40$. The corresponding range of $\tau$
for the path $\gamma$ in \eqref{mm-reg1} is from $1/10$ to~$-2/5$.
\end{example}

\begin{example}
\label{ex3}
The elliptic curve \eqref{mb02} for $k=2$ has conductor $N=24$ and
admits parameterisation by modular units
$$
x(\tau)=\frac{g_1g_{10}g_{11}}{g_2g_5g_7},
\qquad
y(\tau)=-\frac{g_5g_7}{g_1g_{11}}.
$$
Theorem~\ref{MB} applies and produces the evaluation
\begin{align*}
\mm\Bigl(2+x+\frac1x+y+\frac1y\Bigr)
&=\frac1{2\pi}\Bigl(\int_{-1/8}^{i\infty}-\int_{1/8}^{i\infty}\Bigr)
\eta\biggl(\frac{g_1g_{10}g_{11}}{g_2g_5g_7},\frac{g_5g_7}{g_1g_{11}}\biggr)
\\
&=\frac{6}{\pi^2}L(f_{24},2),
\end{align*}
where $f_{24}(\tau):=\eta(2\tau)\eta(4\tau)\eta(6\tau)\eta(12\tau)$,
conjectured in~\cite{B98} and established in~\cite{RZ12}.
\end{example}

\begin{example}
\label{ex4}
For $N=17$, the pair of modular units
$$
x(\tau)=-i\,\frac{g_2g_8}{g_1g_4}, \qquad
y(\tau)=i\,\frac{g_6g_7}{g_3g_5}
$$
parameterise the elliptic curve $i+x+1/x+y+1/y=0$. Applying Theorem~\ref{MB} for $\tau$ ranging from $3/17$ to $-3/17$, we obtain
$$
\mm\Bigl(i+x+\frac1x+y+\frac1y\Bigr)
=\frac{17}{2\pi^2}L(f_{17},2),
$$
where
\begin{align*}
f_{17}(\tau)
:=\frac{q\,(\d x/\d q)}{ix(y-1/y)}
&=q-q^2-q^4-2q^5+4q^7+3q^8-3q^9+2q^{10}
\\ &\qquad
-2q^{13}-4q^{14}-q^{16}+q^{17}+O(q^{18}).
\end{align*}
This Mahler measure evaluation was conjectured in \cite[Table~4]{RV99}.
\end{example}

\begin{example}
\label{ex5}
Another conjecture in \cite[Table~4]{RV99},
$$
\mm\Bigl(\sqrt2+x+\frac1x+y+\frac1y\Bigr)
=\frac{7}{2\pi^2}L(f_{56},2),
$$
corresponds to $k=\sqrt2$ in~\eqref{mb02} and an elliptic curve over $\mathbb Z$ of conductor $N=56$.
The conjecture follows from parameterisation of the curve by the couple
\begin{align*}
x(\tau)
&=\phantom-\frac1{\sqrt2}\,\frac{\eta(\tau)\eta(4\tau)^2\eta(7\tau)\eta(28\tau)^2}{\eta(2\tau)^2\eta(8\tau)\eta(14\tau)^2\eta(56\tau)},
\\
y(\tau)
&=-\frac1{\sqrt2}\,\frac{\eta(2\tau)\eta(4\tau)\eta(14\tau)\eta(28\tau)}{\eta(\tau)\eta(7\tau)\eta(8\tau)\eta(56\tau)},
\end{align*}
so that
\begin{align*}
f_{56}(\tau)
:=\frac{q\,(\d x/\d q)}{\sqrt2\,x(y-1/y)}
&=q+2q^5-q^7-3q^9-4q^{11}+2q^{13}-6q^{17}+8q^{19}
\\ &\qquad
-q^{25}+6q^{29}+8q^{31}+O(q^{34}),
\end{align*}
and integration in Theorem~\ref{MB} for $\tau\in(-15/56,-7/56)\cup(5/56,13/56)$.
\end{example}

It is not clear whether there are finitely or infinitely many cases of the parameter $k$ in~\eqref{mb02}
subject to parameterisation by modular units. A possible approach in cases when such parameterisation
is not available is writing down algebraic relations between any two standard modular units \eqref{mb01} of a given level~$N$
and sieving the relations which may be used in producing the Mahler measures of 2-variable polynomials
which are potentially linked to the wanted Mahler measures by $K$-theoretic machinery~\cite{Co04,La10,LR07}.

Finding what curves $C:P(x,y)=0$ can be parameterised by modular units is an interesting question itself.
F.~Brunault notices some heuristics to the fact that
there are only finitely many function fields $F$ of a given genus~$g$ over $\mathbb Q$
which embed into the function field of a modular curve such that $F$ can be generated by modular units;
for $g\ge2$ this follows from \cite[Conjecture~1.1]{BGGP05}.

\section{3-variable Mahler measures}
\label{s4}

It would be desirable to have an analogue of Theorem~\ref{MB} for 3-variable Mahler measures
of (Laurent) polynomials $P(x,y,z)$ such that the intersection of the zero loci $P(x,y,z)=0$
and $P(1/x,1/y,1/z)=0$ defines an elliptic curve $E$, and $\mm(P)$ is presumably related
to the $L$-series of $E$ evaluated at $s=3$. No example of this type is established, and
one of the simplest evaluations is Boyd's conjecture \cite{B06}
$$
\mm\bigl((1+x)(1+y)-z\bigr)
\overset?=2L'(E_{15},-1)=\frac{225}{4\pi^4}L(E_{15},3).
$$

On the surface $(1+x)(1+y)-z=0$ we have
\begin{align*}
x\wedge y\wedge z
&=x\wedge y\wedge(1+x)(1+y)
=x\wedge y\wedge(1+x)+x\wedge y\wedge(1+y)
\\
&=-x\wedge(1+x)\wedge y+y\wedge(1+y)\wedge x
\\
&=-(-x)\wedge(1+x)\wedge y+(-y)\wedge(1+y)\wedge x.
\end{align*}
Applying the machinery described in
\cite[Section 5.2]{Co04} to the 3-variable polynomial $P(x,y,z)=(1+x)(1+y)-z$ we obtain
$$
m(P)=\frac1{4\pi^2}\int_\gamma\bigl(\omega(-x,y)-\omega(-y,x)\bigr),
$$
where
\begin{equation}
\omega(g,h):=D(g)\,\d\arg h+\frac13\bigl(\log|g|\,\d\log|1-g|-\log|1-g|\,\d\log|g|\bigr)\log|h|
\label{omega}
\end{equation}
and
\begin{align*}
\gamma&:=\{(x,y,z):|x|=|y|=|z|=1\}\cap\{(x,y,z):(1+x)(1+y)-z=0\}
\\ &\qquad
\cap\{(x,y,z):(1+x)(1+y)z-xy=0\}.
\end{align*}
Note that $\{(1+x)(1+y)-z=0\}\cap\{(1+x)(1+y)z-xy=0\}$ is the double cover of an elliptic curve of
conductor 15. Indeed, eliminating $z$ we can write (one half of) its equation as
$$
(1+x_1^2)(1+y_1^2)+x_1y_1=0
$$
in variables $x_1=\sqrt x$, $y_1=\sqrt y$, or
$$
x_2+1/x_2+y_2+1/y_2+1=0
$$
in variables $x_2=x_1y_1$, $y_2=x_1/y_1$. Using the parameterisation of the latter equation
by the modular units from Example~\ref{ex1} we find out that
$$
m(P)=\frac1{2\pi^2}\int_{-1/5}^{1/5}\bigl(\omega(X,Y)-\omega(Y,X)\bigr)
$$
where
$$
X(\tau):=\frac{g_4(\tau)g_7(\tau)}{g_1(\tau)g_2(\tau)}=q^{-2}+O(q^{-1})
\quad\text{and}\quad
Y(\tau):=\frac{g_1(\tau)g_7(\tau)}{g_2(\tau)g_4(\tau)}=1+O(q).
$$
Also note that
$$
1-X(\tau)=-\frac{g_6(\tau)g_7(\tau)}{g_1(\tau)g_3(\tau)}=-q^{-2}+O(q^{-1})
\quad\text{and}\quad
1-Y(\tau)=\frac{g_1(\tau)g_3(\tau)}{g_2(\tau)g_6(\tau)}=q+O(q^2)
$$
are modular units.

The problem with integrating the form \eqref{omega} is that it is, roughly speaking,
integrating the product of \emph{three} modular components: two of them are logarithms of modular functions
(hence of weight 0) and one is the logarithmic derivative of a modular function
(hence of weight 2). On the other hand, the expected data for applying the method from~\cite{RZ12}
used in our proof of Theorem~\ref{MB} in Section~\ref{s2} would be
integrating a product of \emph{two} Eisenstein series of weights $-1$ and~3 (see~\cite{Zu13} for details).

\begin{acknowledgements}
This note would be hardly possible without constant patience of A.~Mellit in
describing details of his work with F.~Brunault. I am pleased to thank Mellit
for all those lectures he delivered to me in person and by e-mail, as well as for providing
me with the sketch~\cite{Me12} of proof of what is stated here as Theorem~\ref{MB}.
I am deeply grateful to F.~Brunault, M.~Rogers and J.~Wan for their helpful assistance
on certain parts of this work. Finally, I thank the anonymous referee for her valuable
comments and healthy criticism that helped to improve the exposition.
\end{acknowledgements}



\begin{thebibliography}{99}

\bibitem{BGGP05}
\textsc{M.\,H.~Baker}, \textsc{E.~Gonz\'alez-Jim\'enez}, \textsc{J.~Gonz\'alez} and \textsc{B.~Poonen},
Finiteness results for modular curves of genus at least~2,
\emph{Amer. J. Math.} \textbf{127} (2005), no.~6, 1325--1387.

\bibitem{Be04}
\textsc{M.\,J.~Bertin},
Mesure de Mahler dune famille de polyn\^omes,
\emph{J. Reine Angew. Math.} \textbf{569} (2004), 175--188.

\bibitem{B98}
\textsc{D.~Boyd},
Mahler's measure and special values of $L$-functions,
\emph{Experiment. Math.} \textbf{7} (1998), no.~1, 37--82.

\bibitem{B06}
\textsc{D.~Boyd},
Mahler's measure and $L$-functions of elliptic curves evaluated at $s=3$,
\emph{Slides from a lecture at the SFU/UBC number theory seminar} (December 7, 2006),
\texttt{http://www.math.ubc.ca/\~{}boyd/sfu06.ed.pdf}\,.

\bibitem{Br08}
\textsc{F.~Brunault},
Beilinson--Kato elements in $K_2$ of modular curves,
\emph{Acta Arith.} \textbf{134} (2008), 283--298.

\bibitem{Co04}
\textsc{J.\,D.~Condon},
\emph{Mahler measure evaluations in terms of polylogarithms},
Dissertation, The University of Texas at Austin (2004).

\bibitem{De97}
\textsc{C.~Deninger},
Deligne periods of mixed motives, $K$-theory and the entropy of certain $\mathbb Z^n$-actions,
\emph{J. Amer. Math. Soc.} \textbf{10} (1997), no.~2, 259--281.

\bibitem{Ka76}
\textsc{N.\,M.~Katz},
$p$-adic interpolation of real analytic Eisenstein series,
\emph{Ann. of Math.} (2) \textbf{104} (1976), no.~3, 459--571.

\bibitem{La10}
\textsc{M.\,N.~Lal\'\i n},
On a conjecture by Boyd,
\emph{Int. J. Number Theory} \textbf{6} (2010), no.~3, 705--711.

\bibitem{LR07}
\textsc{M.\,N.~Lal\'\i n} and \textsc{M.\,D.~Rogers},
Functional equations for Mahler measures of genus-one curves,
\emph{Algebra and Number Theory} \textbf{1} (2007), 87--117.

\bibitem{Me11}
\textsc{A.~Mellit},
Mahler measures and $q$-series,
in: \emph{Explicit methods in number theory} (MFO, Oberwolfach, Germany, 17--23 July 2011),
\emph{Oberwolfach Reports} \textbf{8} (2011), no.~3, 1990--1991.

\bibitem{Me12}
\textsc{A.~Mellit},
Regulator of two modular units formula,
Unpublished note (12 June 2012).

\bibitem{RV99}
\textsc{F.~Rodriguez Villegas},
Modular Mahler measures. I,
in: \emph{Topics in number theory} (University Park, PA, 1997), \emph{Math. Appl.} \textbf{467} (Kluwer Acad. Publ., Dordrecht, 1999), 17--48.

\bibitem{RZ12}
\textsc{M.~Rogers} and \textsc{W.~Zudilin},
{}From $L$-series of elliptic curves to Mahler measures,
\emph{Compositio Math.} \textbf{148} (2012), 385--414.

\bibitem{RZ13}
\textsc{M.~Rogers} and \textsc{W.~Zudilin},
On the Mahler measure of $1+X+1/X+Y+1/Y$,
\emph{Intern. Math. Res. Not.} (to appear); \texttt{doi:\,10.1093/imrn/rns285}.

\bibitem{Sch74}
\textsc{B.~Schoeneberg},
\emph{Elliptic modular functions: an introduction},
translated from the German by J.\,R.~Smart and E.\,A.~Schwandt,
\emph{Die Grundlehren der mathematischen Wissenschaften} \textbf{203} (Springer-Verlag, New York--Heidelberg, 1974).

\bibitem{Ya04}
\textsc{Y.~Yang},
Transformation formulas for generalized Dedekind eta functions,
\emph{Bull. London Math. Soc.} \textbf{36} (2004), 671--682.

\bibitem{Zu13}
\textsc{W.~Zudilin},
Period(d)ness of $L$-values,
in \emph{Number Theory and Related Fields, In memory of Alf van der Poorten},
J.\,M.~Borwein et~al. (eds.),
\emph{Springer Proceedings in Math. Stat.} \textbf{43} (Springer, New York, 2013), 381--395.

\end{thebibliography}
\end{document}